\documentclass{article}

\usepackage[english]{babel}
\usepackage[utf8]{inputenc}  
\usepackage[T1]{fontenc}
\usepackage{enumitem}

\usepackage[letterpaper,top=2cm,bottom=2cm,left=2cm,right=2cm,marginparwidth=1.75cm]{geometry}

\usepackage{amsmath,amssymb}
\usepackage{amsthm}
\usepackage{graphicx}
\usepackage[colorlinks=true, allcolors=blue,pdfencoding=auto, psdextra]{hyperref}
\newtheorem{theorem}{Theorem}
\newtheorem{lemma}{Lemma}

\newtheorem{remark}{Remark}

\newcommand{\R}{\mathbb{R}}

\newcommand{\C}{\mathbb{C}}
\renewcommand{\H}{\mathbb{H}}
\renewcommand{\O}{\mathbb{O}}
\newcommand \<{\langle}
\renewcommand \>{\rangle}
\newcommand{\SO}{\mathrm{SO}}
\newcommand{\Sp}{\mathrm{Sp}}
\newcommand{\Spin}{\mathrm{Spin}}
\newcommand{\SU}{\mathrm{SU}}
\newcommand{\ri}{\mathrm{i}}

\newcommand{\weg}[1]{}

\newcommand{\const}{\textrm{const}}

\title{On Separation of Variables for  Symmetric Spaces of Rank~1}
\author{Alexey Bolsinov\footnote{Department of Mathematical Sciences,
 Loughborough University,
 LE11 3TU, UK  and Institute of Mathematics and Mathematical Modeling, Almaty, Kazakhstan\ \ \quad {\tt A.Bolsinov@lboro.ac.uk}}, Holger R.~Dullin\footnote{School of Mathematics and Statistics, University of Sydney, Australia \ \ \quad {\tt Holger.Dullin@sydney.edu.au}}, Vladimir. S. Matveev\footnote{
Institut f\"ur Mathematik, Friedrich Schiller Universit\"at Jena,
07737 Jena,  Germany \ \quad ORCID: 0000-0002-2237-1422 \ \ \quad {\tt  vladimir.matveev@uni-jena.de}}, Yury Nikolayevsky\footnote{Department of Mathematical and Physical Sciences, La Trobe University, Melbourne, Victoria, 3086, Australia
\ \ \quad {\tt  y.nikolayevsky@latrobe.edu.au}}}
\date{March 2025}

\begin{document}

\maketitle

\begin{abstract}
    We study existence and nonexistence of diagonal and separating coordinates for Riemannian  symmetric spaces of rank~1.  We generalize the results of Gauduchon and Moroianu, 2020, by showing that a symmetric space of rank~1 has diagonal coordinates if and only if it has constant sectional curvature. This implies that orthogonal separation of variables on a symmetric space of rank 1 is possible only in the constant sectional curvature case. We show that on the complex projective space $\C P^n$  and on complex hyperbolic space $\C H^n$, with $n\ge 2$, separating coordinates necessarily have precisely $n$ ignorable coordinates. 
    In view of results of Boyer et al, 1983 and 1985, and later results of Winternitz et al, 1994,  this completes the description of separation of variables on $\C P^n$ for all $n$ and on  $\C H^n$ for $n=2,3$. 


\end{abstract}

\section{Introduction}
All geodesics on Riemannian  compact rank~1 symmetric spaces are closed and have the same length. The three families of such manifolds are the standard sphere $S^n$,  the complex projective  space $\C P^n$, and the quaternionic projective  space $\H P^n$ corresponding to the three associative division algebras $\R$, $\C$, and Hamilton's quaternions~$\H$.  The final isolated example is $\O P^2$ related to non-associative octonions. \weg{ $\C P^n$ and $\H P^n$ can be   obtained by Hopf fibrations as
\begin{align*}
     S^{2n+1}/S^1 & = \C P^n, \\
     S^{4n+3}/S^3 &=  \H P^n \,.
\end{align*} }
As symmetric spaces they are given by the following quotients
\begin{align*}
    \SO(n+1)/\SO(n) & =   S^n, \\
    \SU(n+1)/\mathrm{S}( \mathrm{U}(n) \times \mathrm{U}(1) ) & =  \C P^n, \\
    \Sp(n+1)/(\Sp(n) \times \Sp(1)) & =  \H P^n, \\
    \mathrm{F}_4/\Spin(9) & =  \O P^2 \,.
\end{align*}
\weg{The standard metric for the complex and quaternionic projective spaces is called the Fubini-Study metric; it is  constructed from homogeneous coordinates in $\C^{n+1}$ respectively $\H^{n+1}$.}

These spaces have  noncompact   twins, which are rank~1 Riemannian symmetric spaces of negative curvature:   
\begin{align*}
    \SO(n,1)/\SO(n) & =  H^n   &  \textrm{(real) hyperbolic space}  \\
    \SU(n,1)/\mathrm{S}( \mathrm{U}(n) \times \mathrm{U}(1) ) & = \C H^n  &   \textrm{complex hyperbolic space}  \\
    \Sp(n,1)/(\Sp(n) \times  \Sp(1)) & =  \H H ^n   &   \textrm{quaternionic hyperbolic space}  \\
    \mathrm{F}^{-}_4/\Spin(9) & =  \O H^2  &   \textrm{Cayley hyperbolic plane}
\end{align*}


The results of this paper can be presented in the following table, where ``diagonal'' means that there exist (or do not exist) local coordinates, in which the metric tensor is diagonal, and ``separable'' means that there exist (or do not exist) local coordinates, in which the geodesic equation and the Laplacian are separable, see definitions in Section \ref{sec:separation}. In the table below, we assume that $n\ge 2$.
\begin{center}
\begin{tabular}{ c |  c | c }
        & diagonal & separable  \\ \hline
 $ S^n$    & yes & orthogonal \cite[Theorem 3.3]{Kalnins_book} \\ 
  $H^n$  &   yes & orthogonal \cite[Theorem 5.1]{Kalnins_book} \\ 
 $\C P^n$ & no \cite[Props~3.1, 3.2]{gauduchon2020non} & non-orthogonal, $n$ ignorable variables (Theorem~\ref{thm:cpn1})\\  
  $\C H ^n$ & no (Theorem~\ref{th:nondiag}) & non-orthogonal, $n$ ignorable variables (Theorem~\ref{thm:cpn1})\\  
 $\H P^n$ & no \cite[Prop 4.1]{gauduchon2020non} & no   (Theorem~\ref{thm:3})     \\
$ \H H ^n$,  $\O P^2$, $\O H^2$   &    no (Theorem~\ref{th:nondiag}) & no (Theorem~\ref{thm:3}) \\ 
\end{tabular}
\end{center}

In more detail, we  study $\C P^n, $ $\C H ^n$, $\H  P^n$, $ \H H ^n$, $\O P^2$  and 
$\O H^2$ from the viewpoint of local existence of diagonal and separating coordinates. The nonexistence of diagonal coordinate systems on  $\C P^n $  and $\H  P^n$   was established in \cite{gauduchon2020non}.
We generalize these results to $\C H ^n$,  $ \H H ^n$,   $\O P^2$ and  $\O H^2$ in Theorem~\ref{th:nondiag}. The proof for $\C H ^n$  and  $ \H H ^n$ follows the same ideas as in \cite{gauduchon2020non}.  

We next proceed to the study of separation of variables on these spaces; see Section~\ref{sec:separation} for definitions. The nonexistence of orthogonal separating coordinates follows from the nonexistence of diagonal coordinates. On $S^n$ and $H^n$ non-orthogonal separation coordinates exist, but they can always be transformed into orthogonal ones. We show in Theorem~\ref{thm:3}   that for  $n\ge 2$, the spaces   $\H  P^n$,   $\H  H^n$,  $\O P^2$  and  $\O H^2$  admit no separation of variables. 

Non-orthogonal separating variables exist on $\C P^n$ and on $\C H^n$, see  \cite{BKW1983,BKW1985,ORW1993,WOR1994} where   several families  of  examples have been discussed. All these examples have precisely $n$ ignorable coordinates. On the other hand, all separating coordinate systems on  $\C P^n$ with $n$ ignorable coordinates have been classified  in~\cite{BKW1985}.  Moreover, it was shown in \cite{BKW1985} that on $\C P^2$, any separation of variables has two ignorable coordinates, which completes the classification of separations of variables for this space.   
 
Non-orthogonal separating coordinates on $\C H^2$ with two ignorable coordinates were explicitly constructed in \cite{BKW1983}. Non-orthogonal separating coordinates on $\C H^3$ having three ignorable coordinates were explicitly constructed in \cite{WOR1994}. In general, the picture for the space $\C H^n$ is substantially more complicated than that for its compact dual $\C P^n$, since the isometry algebra $\mathfrak{su}(n,1)$ contains $n+2$ pairwise non-conjugate abelian subalgebras of dimension $n$ (not just a Cartan subalgebra!), see~\cite[Theorem~5.1]{ORWZ1990}.  
To describe separating coordinate systems with $n$ ignorable coordinates, one has to study whether a given $n$-dimensional abelian subalgebra of $\mathfrak{su}(n,1)$ indeed generates ignorable coordinates in a non-orthogonal separation of variables, and if so, which ones. This is quite a nontrivial task which has been completed only for $n=2, 3$ \cite{WOR1994}. Two out of $n+2$ pairwise non-conjugate $n$-dimensional abelian subalgebras of $\mathfrak{su}(n,1)$ are Cartan and are relatively easy to handle, for all $n \ge 2$. However, the remaining $n$ have been treated in full detail only for $n=2,3$, and the complete description for $n=3$ obtained in \cite{WOR1994} is already quite involved. 

We show in Theorem \ref{thm:cpn1}  that every separating coordinate system on $\C P^n$ and on $\C H^n$ has precisely $n$ ignorable coordinates (as we mentioned above, for  $\C  P^2$, this fact is established in \cite[\textsection~6B]{BKW1985}). This implies that the classifications of separations of variables on $\C P^n$ and on $\C H^2$ and $\C H^3$ presented in   \cite{BKW1985,ORW1993,WOR1994} are complete.

Let us now comment on why we have chosen symmetric spaces for our investigation. It is known since St\"ackel \cite{Staeckel1893,Staeckel1897} and Levi-Civita \cite{Levi-Civita1904} that separating coordinates are closely related to Killing tensors of order one  and two. Symmetric spaces of rank~1 have  a large algebra of Killing tensors of order one (of Killing vector fields) and also a large algebra of Killing tensors of order two, and so they are in a certain sense natural candidates for the existence of separating coordinates. We also hope that the methods developed in this paper, combined with those from \cite{gauduchon2020non}, will allow to study separating and diagonal coordinates on symmetric spaces of higher rank. We would like to emphasize, that by \cite{MN2024}, not all Killing tensors of order two on $\H P^n$ and $\O P^2$ are quadratic polynomials of the Killing vector fields, and so purely algebraic methods to study separation of variables using universal enveloping algebra will not be in general sufficient. We also note that by the results of \cite{MN2025}, the study of Killing tensors of reducible symmetric spaces is reduced to that of irreducible components, which may substantially facilitate the study of separating coordinates. 

Geodesics of rank~1 symmetric spaces are completely and explicitly described in e.g.~\cite{Besse}, in particular in the compact case they are closed and have the same length. The interest in separation of variables goes way beyond the description of the solution of the geodesic equations. In particular, in irreducible symmetric spaces the Eisenhart-Robertson condition (see e.g., \cite[Sec.~2]{Eisenhart1934}) is automatically fulfilled so that the Helmholtz equation also separates in the separating coordinate systems. In addition, it is easy to introduce potential energy in the picture, such that the corresponding Hamiltonian system and the corresponding Helmholz equation still separate. 

Separations of variables was studied and used since the middle of the 19th century. The problem of describing and classifying, up to isometries, all separations of variables in the space forms has been solved by Eisenhart \cite{Eisenhart1934} in small dimensions and under certain nondegeneracy assumptions. The solution for Riemannian space forms completed in \cite{Kalnins_book}. However, the pseudo-Riemannian case is still open: while the orthogonal case has been completed in \cite{BKM2025, KMR1984}, the description of non-orthogonal separations for constant curvature metrics of indefinite signature is unknown. The present paper solves the natural analog of the Eisenhart problem for all rank~1 symmetric spaces except for $\C H ^n $ with $n\ge 4$, and reduces the remaining cases to those which can be solved with computer algebra.  

\subsubsection*{Acknowledgements. } A.\,B. was supported by the Ministry of Science and Higher Education of the Republic of Kazakhstan (grant No. AP23483476). V.\,M.\ was supported by the DFG (projects 455806247 and 529233771) and by the Friedrich Schiller University Jena. 
H.\,D., V.\,M. and Y.\,N. were supported  by the ARC  Discovery grant DP210100951.  Most results were obtained during a research visit of A.\,B. and V.\,M. to La Trobe University, Melbourne, and to the Sydney Mathematical Research Institute (SMRI) supported by the SMRI International Visitor Program. The research project was initiated and first results were obtained during the Minor Program at the MATRIX Research Institute, Australia  in which all four authors  participated. 

We are very grateful to  C. Chanu and G. Rastelli for valuable comments and for giving us useful references on non-orthogonal separation of variables. 



\section{Separation of variables} \label{sec:separation}
\weg{
The concept of separation of variables plays a fundamental role in mathematical physics and is often introduced to students in their first year of university. 
The method enables the reduction of some partial differential equations into independent one-dimensional ordinary differential equations 
facilitating their solution and analysis.}

In this paper, we study separation of variables for the geodesic equations 
and adopt the formal definition of separation of variables which has been already introduced by Levi-Civita \cite{Levi-Civita1904}.
The geodesic flow with metric $g_{ij}$ is described through the Hamiltonian $H(x,p)=\tfrac12 \sum g^{ij}(x) p_i p_j$ on the cotangent bundle $T^*M$.
By separation of variables on an $N$-dimensional manifold  $M^N $
we understand the existence  of a  function $W(x_1,\dots ,x_N, c_1,\dots,c_N)$ of $2 N$  variables 
such that the following conditions are fulfilled:
\begin{enumerate}[label=(\alph*),ref=\alph*]
    \item \label{Wns} 
    The $N\times N$-matrix $\tfrac{\partial^2 W}{\partial c_i \partial x_j}$ is non-singular. 

    \item \label{Hc1} 
    $H(x,\tfrac{\partial W}{\partial x}) = c_1$  (the Hamilton-Jacobi equation).
    
    \item \label{Wi}
    $W(x,c)= \sum_{i=1}^N W_i(x_i,c_1,\dots, c_N)$ (where the function $W_i$ depends on $x_i$ and $c$ only). 
\end{enumerate}

It is well known, see e.g. \cite[\textsection 1]{Kalnins_book}, that the existence of a function $W$ satisfying conditions~\eqref{Wns} and~\eqref{Hc1} (such a function is called a generating function)  allows one to construct a local  coordinate system $(c_1,\dots,c_N, Q_1,\dots,Q_N)$  on the cotangent bundle  $T^*M$ such that $c_1= H(x,p)$ and such that in this coordinate system,  the standard symplectic form  $\sum_{i=1}^N dp_i\wedge dx_i$ has the ``canonical''  form $\sum_{i=1}^N dc_i \wedge d Q_i $.  Indeed, consider the following two local mappings: 
 \begin{equation} \label{eq:2}  \begin{array}{cc}  \phi: \mathbb{R}^{2N}(x,c) \to \mathbb{R}^{2N}(x,p) , & \phi(x,c) =\left(x, \tfrac{\partial W}{\partial x}\right) \\ 
\psi: \mathbb{R}^{2N}(x,c) \to \mathbb{R}^{2N}(Q,c), & \psi(x,c) =\left(\tfrac{\partial W}{\partial c}, c\right).  \end{array}\end{equation} 
By~\eqref{Wns}, the mappings $\phi$, $\psi$, and therefore $\psi \circ \phi^{-1} $, are local  diffeomorphisms. 
In view of the equation 
$$
dW= \sum_i \tfrac{\partial W}{\partial x_i} dx_i +  \sum_i \tfrac{\partial W}{\partial c_i} dc_i = \left( \textrm{after applying $\phi, \psi$} \right)=  \sum_i p_i dx_i + \sum_i Q_idc_i,  
$$
the equation $0=d(dW)$ is equivalent to  $\sum_i d(p_idx_i)= \sum_i d(c_idQ_i)$, as claimed.     
As $H=c_1$ by~\eqref{Hc1}, the Hamiltonian system generated by $H$ looks extremely simple in coordinates $(c, Q)$  and its general solution is given by 
$(c(t), Q(t))= ( \const_1, \const_2,\dots,\const_N, {  \textrm{Const}_1- t},\textrm{Const}_2,\dots,     \textrm{Const}_N).$

Unfortunately, a function $W$ satisfying (\ref{Wns}, \ref{Hc1}, \ref{Wi}) exists not for many coordinate system; moreover, for most metrics, the required coordinates do not exist at all.  Finding coordinates $x_1,\dots,x_N$, for a given metric for which there exists a function $W$ satisfying (\ref{Wns}, \ref{Hc1}, \ref{Wi}) is a nontrivial task. Such coordinates are called {\it separating coordinates} or {\it separating variables} in our paper, and we study their existence and construction for symmetric spaces of rank~1. Probably, the only effective way to find such coordinates, which we follow, is based on the relation of the separating coordinates to Killing tensors of order one and two due to St\"{a}ckel \cite{Staeckel1893,Staeckel1897} and Levi-Civita \cite{Levi-Civita1904}.

First observe that the coordinates $c_1,\dots,c_N$ viewed as functions on $T^*M$ are functionally independent and Poisson-commute. These functions are called {\it separation constants}.  It is known that they are necessarily either linear or quadratic  in momenta $p_1,\dots,p_N$, and so they correspond to Killing tensors of order one or two. A necessary and sufficient condition for a system of $r$ Killing vectors and $N-r$ Killing tensors on a Riemannian manifold to correspond to a separation of variables is given in the following theorem.
 
\begin{theorem}[follows from Theorem 2.7 of \cite{BF1980} or \textsection 7 in \cite{Benenti1990},  Theorem 4 in \cite{KalninsMiller1981} and \cite{agafonov}, see also \cite{BCR2001, BCR2002}] \label{th:ben} 
Let $M^N$ be a Riemannian manifold equipped with a set of $r$ Killing vector fields $V_{N-r+1},\dots, V_N$ and  $N-r$ quadratic Killing tensor fields $\overset{1}{K},\dots,\overset{N-r}{K}$. There locally exists a  separating coordinate system $(x_1,\dots,x_{N-r}, t_{N-r+1}, \dots, t_N)$  on $M^N$ such that the separation constants $c_1,\dots,c_N$  are the linear and the quadratic in momenta functions corresponding to the given Killing vectors and Killing tensors if and only if  the following properties hold:

\begin{enumerate}[label=\Roman*.]
\item  \label{1}  The linear and the quadratic in momenta functions corresponding to the Killing vectors and the Killing tensors, respectively, Poisson commute and are functionally independent. 

\item  \label{2} One of the Killing tensors of order two is the metric tensor.

\item   \label{3} Locally, the quadratic Killing tensors admit $N-r$ mutually orthogonal eigenvector fields $X_i$, $\ i = 1,\dots, N-r$, which are orthogonal to all Killing vector fields $V_{N-r+1},\dots, V_N$.  
\end{enumerate}

Moreover, in the separating coordinate system, the 
Killing vector fields $V_j$ are constant linear combinations of the  vector fields  $\partial_{t_i}, \ i= N-r+1,\dots,N$,  and 
the separation constants $c_1,\dots,c_{N-r}$ corresponding to the quadratic Killing tensors are given by 
\begin{equation} \label{eq:metric}
    c_\ell =   \sum_{\alpha,\beta=1}^{N-r} \overset{\ell}{{k}^{\alpha\beta}}(x_1,..\dots,x_{N-r})p_\alpha p_\beta 
       + \sum_{i,j=N-r+1}^N \overset{\ell}{{h}^{ij}}(x_1,\dots,x_{N-r}) \, p_i p_j, 
    \end{equation}
    and the Hamiltonian of the geodesic flow is $c_1/2$. 
For any $i,j \in {N-r+1},\dots,N$,  the functions 
 $\sum_{\alpha\beta=1}^{N-r} \overset{\ell}{k^{\alpha\beta}} p_\alpha p_\beta + \overset{\ell}{h^{ij}}$,  $\ell=1,\dots,N-r$, viewed as function 
 on  the cotangent bundle to $N-r$-dimensional manifolds  with  local coordinates  $x_1,\dots, x_{N-r}$, are given by the St\"ackel formula. 
\end{theorem}

The coordinates $t_{N-r+1},\dots, t_N$ from Theorem~\ref{th:ben}  are called {\it ignorable coordinates}.

Although it is well known, let us recall the St\"ackel formula
following  \cite{Eisenhart1934, disser}. Take a  non-singular $(N-r)\times (N-r)$ matrix $S= (S_{ij})$  with $S_{ij}$ being a
function  of the $i$-th variable $x_i$ only. Next, consider the functions $c_\alpha$, $\alpha =1,\dots, N-r$,  given by the following system of linear equations
 \begin{equation} 
 \label{eq:St_intro}
S { C}  = {P},  
\end{equation} 
where  ${C}= \left(c_1, c_2,\dots ,c_{N-r}\right)^\top$ and $ {P}= \left(f_1(x_1)p_1^2 + \phi_1(x_1), f_2(x_2)p_2^2 + \phi_2(x_2),\dots ,f_{N-r}(x_{N-r})p_{N-r}^2 + \phi_{N-r}(x_{N-r})\right)^\top$.  It is known that the functions $c_i$ are in involution, and  the coordinates $x_i$ are separating for the Hamiltonain system corresponding to  any of them, or even  to all of them together.   

 Theorem \ref{th:ben} can be understood, in the Riemannian case,  as an equivalent reformulation of the definition of separating coordinates and of  their existence. 
 Though this reformulation is sufficient for us,  we give the formula for the functions $W_i$, $i =1,\dots, N-r$, as
 $$
 W_i(x_i, c_1,\dots, c_N)= \pm \int_0^{x^i}  \sqrt{    \tfrac{1}{f_i(\xi)}\left( - \phi_i(\xi) + \sum^{N-r}_{\alpha=1} S_{i \alpha}(\xi) c_\alpha\right) } d\xi,  
 $$
and for $i=N-r +1,\dots, N$, we take $W_i= c_i$.

\begin{remark}
   Theorem~\ref{th:ben} is shorter and is visually different from that stated, e.g., in \cite{Benenti1990}   and  \cite{KalninsMiller1981}. This difference is due to the assumption on the metric being positive definite, which prohibits the coordinates of the second class in the terminology of Benenti.  Moreover, it uses a small improvement obtained in \cite{agafonov}.  
\end{remark}

\begin{remark} \label{rem:fromben}
From Theorem~\ref{th:ben} we easily obtain the following:
\begin{itemize}
    \item 
    If a Riemannian manifold admits separation of variables such that none of the separation constants $c_1,\dots , c_N$, correspond to Killing vector fields, i.e., $r=0$ in the notation of Theorem \ref{th:ben},   then it locally admits an orthogonal coordinate system.

\item  In the coordinates $(x_1,\dots,x_{N-r},t_{N-r+1},\dots,t_N)$, 
from Theorem \ref{th:ben}, the submanifolds corresponding to the coordinates $x_1,\dots,x_{N-r}$ are totally geodesic, and are  orthogonal at every point to the pairwise commuting Killing vector fields  $\tfrac{\partial }{\partial t_{N-r+1}},\dots,\tfrac{\partial }{\partial t_{N}}$.  The metric is given by 
    \begin{equation} \label{eq:coords}
        ds^2 = \sum_{\alpha, \beta= 1}^{N-r} g_{\alpha \beta }(x) \, dx_\alpha dx_\beta + \sum_{i,j=N-r+1}^{N} h_{ij}(x) \, dt_i dt_j.
    \end{equation}
\weg{\item 
    If a Riemannian manifold admits a non-orthogonal separation of variables, then there is a local coordinate system $(x, t)$, where $x = (x^1, \dots, x^m), \, t = (t^1, \dots, t^\ell), \; m, \ell \ge 1$, such that the metric is given by 
    \begin{equation} \label{eq:coords}
        ds^2 = \sum_{\alpha, \beta= 1}^m g_{\alpha \beta }(x) \, dx_\alpha dx_\beta + \sum_{i,j=1}^{\ell} h_{ij}(x) \, dt_i dt_j.
    \end{equation}
    One easily sees that the submanifolds $t = \mathrm{const}$ are totally geodesic, with $x^i$ being orthogonal coordinates on them. The submanifolds $x = \mathrm{const}$ are flat (but not totally geodesic), and the vector fields $\tfrac{\partial}{\partial t_i}$ are pairwise commuting Killing vector fields.}
\end{itemize}
\end{remark}

\begin{remark} \label{rem:3} A lot of  research related to separation of variables on $\C P^n$ and $\C H^n$,   see e.g. \cite{WOR1994, ORW1993}, 
was done in the context of superintegrability and multi-separation of variables, see, e.g., \cite{KKM18}.  In smaller dimensions, after symplectic reduction with respect to ignorable coordinates, the geodesic flow of the metric and the corresponding Killing tensors produce a superintegrable system on the space whose commuting integrals correspond to the first block of \eqref{eq:metric}  with potentials which are essentially the components of the second block of   \eqref{eq:metric}.     We comment on this in Section~\ref{sec:con}.  
\end{remark}

\section{Nonexistence of separating coordinates on 
 \texorpdfstring{ $\H  P^n$}{HP\unichar{"006E}}, \texorpdfstring{ $\H  H^n$}{HP\unichar{"006E}}, \texorpdfstring{$\O P^2$}{OP\unichar{"00B2}}  and 
  \texorpdfstring{$\O H^2$}{OH\unichar{"00B2}} }

We start with the following fact, a substantial part of which follows from the work of Gauduchon and Moroianu~\cite{gauduchon2020non}.
\begin{theorem} \label{th:nondiag}
    A rank~1 Riemannian symmetric space locally admits a coordinate system in which the metric has diagonal form if and only if it has constant curvature.
\end{theorem}
\begin{proof}
    The `if' part is well known. We only need to establish the `only if' claim.

    By~\cite[Propositions~3.1, 3.2, 4.1]{gauduchon2020non}, no local diagonal metric exists on either $\C P^n$ or $\H P^n$ for $n > 1$. The proofs in~\cite{gauduchon2020non} only use the algebraic properties of the curvature tensor and the local behavior of the complex structure (for $\C P^2$), and so almost verbatim work for the non-compact spaces $\C H^n$ and $\H H^n$ with $n > 1$.  

    Suppose that there is a local orthogonal coordinate system $x^i, \, i = 1, \dots, 16$, on $\O P^2$ in which the metric has a diagonal form. Relative to such a coordinate system, the components of the curvature tensor satisfy the property $R_{ijkl}=0$, for all pairwise non-equal $i,j,k$ and $l$. At a point $o \in \O P^2$, define an orthonormal basis $\{e_i\}$ for $T_o \O P^2$ such that $e_i$ is a multiple of $\partial/\partial x^i$, for $i = 1, \dots, 16$. Then $R(e_1,e_2,Z,W)=0$, for all $Z,W \in T_o \O P^2$ such that $Z, W \perp e_1, e_2$. 

We can identify $T_o \O P^2$ with the ($\mathbb{R}$-)\,linear space $\O^2$ via a linear isometry, so that a vector $X \in T_o \O P^2$ is represented as $X=(x_1,x_2),\, x_1, x_2 \in \O$. Under this identification, the curvature tensor of $\O P^2$ of sectional curvature between $1$ and $4$ is given in \cite[Equation~6.12]{BG}. For $X=(x_1, x_2), \, Y=(y_1, y_2), \, Z=(z_1, z_2) \in T_o\O P^2 = \O^2$ we have:
\begin{multline*}
R(X,Y)Z=
(4\<x_1,z_1\>y_1-4\<y_1,z_1\>x_1-(z_1y_2)x_2^*+(z_1x_2)y_2^*-(x_1y_2-y_1x_2)z_2^*,\\
4\<x_2,z_2\>y_2-4\<y_2,z_2\>x_2-x_1^*(y_1z_2)+y_1^*(x_1z_2)+z_1^*(x_1y_2-y_1x_2)),
\end{multline*}
where ${}^*$ is the octonion conjugation and $\<u,v\>= \mathrm{Re} (uv^*)$, for $u, v \in \O$.

As the isotropy group $\Spin(9)$ acts transitively on the unit sphere of $T_x \O P^2$, we can take $e_1 = (1, 0)$, and then $e_2 = (y_1, y_2)$, with $y_1, y_2 \in \O$ and $y_1 \perp 1$. Then for $Z,W \perp e_1, e_2$, the condition $R(e_1,e_2,Z,W)=0$ gives
\begin{equation} \label{eq:octoR}
    \<y_2 z_2^*,w_1\> +\<2y_1z_2-z_1^*y_2,w_2\> = 0,
\end{equation} 
where we used the fact that $y_1 \perp 1$, and so $y_1^* = -y_1$. Take $w_2=0$. Then $w_1 \perp 1, y_1$, and so~\eqref{eq:octoR} gives $y_2 z_2^* \in \mathrm{Span}(1, y_1)$. Assuming $y_2 \ne 0$ we arrive at a contradiction, as the left multiplication by a nonzero octonion is injective and as $z_2$ can be chosen arbitrarily from the $7$-dimensional space $y_2^\perp \cap \O$. It follows that $y_2 = 0$, and then~\eqref{eq:octoR} gives $\<y_1z_2,w_2\> = 0$, again leading to a contradiction, since $y_1 \ne 0$ and as $z_2, w_2 \in \O$ can be chosen arbitrarily.  

The same argument works for the octonion hyperbolic plane $\O H^2$, since under a linear isometry between the tangent spaces to $\O P^2$ and $\O H^2$, their curvature tensors differ only by the sign.
\end{proof}

From Remark~\ref{rem:fromben} and Theorem~\ref{th:nondiag} it easily follows that no rank~1 symmetric space of non-constant curvature admits an \textit{orthogonal} separation of variables.

\medskip

We next address \textit{non-orthogonal} separation of variables and prove the following.
\begin{theorem} \label{thm:3}
    The Riemannian symmetric spaces $\H P^n \, (n \ge 2), \H H^n \, (n \ge 2), \O P^2$ and $\O H^2$ admit no local non-orthogonal separation of variables.
\end{theorem}
\begin{proof}
By Theorem~\ref{th:ben} (see also~\eqref{eq:coords}), the local existence of non-orthogonal separation of variables on a Riemannian space $M$ of dimension $N$, implies the local existence of two complementary, orthogonal local foliations on $M$, with the first consisting of totally geodesic submanifolds of dimension $N-r > 0$ admitting a diagonal metric, and the second, consisting of flat submanifolds of dimension $r$ which are orbits of an abelian $r$-dimensional subgroup $K$ of the isometry group of $M$. The picture here is very similar to that for the polar action of the group $K$ on $M$. Recall that a proper action of a group $K$ on a complete, connected Riemannian manifold $M$ is called \textit{polar}, if it admits a \textit{section}, that is, an embedded, totally geodesic submanifold $\Sigma$ which meets all the orbits of $K$, and which intersects all the orbits orthogonally in each of its points (for the modern state of the theory of polar actions, the reader is referred to \cite{DiazRamos}, \cite{DVK} and the bibliographies therein). In our case, all the assumptions are local, and we do not see how we can guarantee the properness of the action of $K$ (or, a priori, even the closedness of $K$). However, we can make use of some results of the theory of polar actions. 

Suppose $M=G/H$ is our symmetric space, being acted upon by an abelian group $K \subset G$, locally admitting a totally geodesic section $\Sigma$ of complementary dimension. Denote $\mathfrak{g}$ and $\mathfrak{k}$ the Lie algebras of $G$ and $K$, respectively. Assume that the point $o \in M$, the projection of the identity of $G$, is a regular point of the action of $K$. Denote $\mathfrak{m} = T_o\Sigma \subset T_oM$, and let $\mathfrak{m}^\perp = T_o(Ko)$ be its orthogonal complement.  

{
\begin{lemma} \label{l:tgcomplementary}
    With the above assumptions and notation, the following holds:
    \begin{enumerate}[label=\emph{(\alph*)},ref=\alph*]
        \item \label{it:lts} 
        The subspace $\mathfrak{m}^\perp \subset T_oM$ is a Lie triple system and hence is tangent to a totally geodesic submanifold of $M$ passing through $o$.
        
        \item \label{it:bracketorth}
        The subspace $[\mathfrak{m}, \mathfrak{m}]$ is orthogonal to $\mathfrak{k}$ relative to the Killing form of $\mathfrak{g}$.
    \end{enumerate}
\end{lemma}
\begin{proof}
    To prove assertion~\eqref{it:lts}, we compute the curvature tensor $R$ of $M$ in the notation of formula~\eqref{eq:coords}. Note that the subspaces $\mathfrak{m}$ and $\mathfrak{m}^\perp$ are spanned, respectively, by $\partial_{x_\alpha}, \, \alpha=1,\dots,N-r$, and by $\partial_{t_i}, \, i=N-r+1,\dots,N$. A direct computation shows that $R(\mathfrak{m}^\perp,\mathfrak{m}^\perp)\mathfrak{m}^\perp \subset \mathfrak{m}^\perp$. It follows that $\mathfrak{m}^\perp$ is a Lie triple system, and hence is tangent to a totally geodesic submanifold of $M$, by Cartan's Theorem~\cite[Theorem~7.2]{Helgason}.

    The proof of assertion~\eqref{it:bracketorth} is verbatim the proof of assertion (ii) of the Proposition in~\cite[p.~195]{Gorodski} --- it only requires local arguments.
\end{proof}
}

Note that the totally geodesic submanifold in Lemma~\ref{l:tgcomplementary}\eqref{it:lts} is \textit{not}, in general, the orbit of $K$, a leaf of the $(N-r)$-dimensional flat complementary foliation. 

\medskip

Assume there is a local non-orthogonal separation of variables on the space $M=\H P^n$ or $M=\H H^n, \, n > 1$. The arguments for both spaces are similar, so let us assume that $M =\H P^n$. By Lemma~\ref{l:tgcomplementary}\eqref{it:lts} and Theorem~\ref{th:nondiag}, the tangent space to $M$ at a regular point of the action of $K$ is the orthogonal sum of two Lie triple systems, one of which is tangent to a totally geodesic submanifold of constant curvature. By~\cite[Theorem~1]{Wolf}, the maximal totally geodesic submanifolds of $\H P^n$ are $\H P^k, \, k<n$, and $\C P^n$, and so the only possible case is that one of the totally geodesic submanifolds is $\Sigma = \H P^1 = S^4$ (and the other one is then $\H P^{n-1}$). This means that $m = 4$, and through every regular point $x$ there passes a flat submanifold $L = Kx \subset M$ orthogonal to the totally geodesic $\H P^1$. But then the tangent space to $L$ at each point is invariant relative to the quaternionic structure, and hence $L \subset M$ is a quaternionic submanifold. By~\cite[Theorem~5]{Gray}, $L$ must be totally geodesic, which contradicts the fact that it is flat. 

\smallskip

The fact that the octonionic projective plane $\O P^2 = \mathrm{F}_4/\Spin(9)$ admits no non-orthogonal separation of variables follows from the dimension count. By~\cite[Theorem~1]{Wolf} the maximal dimension of a proper, totally geodesic submanifold of $\O P^2$ is $8$, while the maximal dimension of an abelian subspace of the algebra $\mathfrak{f}_4$ is  $4$, as any such subspace lies in a Cartan subalgebra of $\mathfrak{f}_4$. 

\smallskip

The same simple argument does not, unfortunately, work for the octonionic hyperbolic plane $\O H^2 = \mathrm{F}^-_4/\mathrm{Spin}(9)$, as the noncompact Lie algebra $\mathfrak{f}^-_4$ admits abelian subalgebras not lying in any Cartan subalgebra, and having the dimension greater than the rank (up to at least $8$; to see this, we note that a solvable group whose Lie algebra is a $1$-dimensional extension of a $2$-step nilpotent $15$-dimensional algebra $\mathfrak{v} \oplus \mathfrak{z}$ with the center $\mathfrak{z}$ of dimension $7$ acts simply transitively on $\O H^2$; a required $8$-dimensional abelian subalgebra of $\mathfrak{f}^-_4$ can be taken as the direct sum of $\mathfrak{z}$ and a line in $\mathfrak{v}$).

By Lemma~\ref{l:tgcomplementary}\eqref{it:lts} and Theorem~\ref{th:nondiag}, the tangent space to $\O H^2$ at a regular point is the orthogonal sum of two Lie triple systems, one of which is tangent to a totally geodesic submanifold of constant curvature. From~\cite[Theorem~1]{Wolf}, the maximal totally geodesic submanifolds of $\O H^2$ are $\H H^2$ and the real hyperbolic space $H^8$, and so the only possible case is that one of the totally geodesic submanifolds is $\Sigma = H^8$ (and the other one is also $H^8$). 

We use the presentation given in~\cite[Section~4]{Kollross}. We have the decomposition $\mathfrak{f}^-_4 = \mathfrak{so}(8) \oplus \O \oplus \O \oplus \O$ into linear subspaces orthogonal relative to the Killing form, and $T_o \O H^2 = \{0\} \oplus \{0\} \oplus \O \oplus \O$. The Lie bracket on $\mathfrak{f}^-_4$ is given by
\begin{equation*}
    [(A, u,v,w),(B,x,y,z)] = (C,r,s,t),
\end{equation*}
where $u,v,w,x,y,z,r,s,t \in \O$ and $A, B, C \in \mathfrak{so}(8)$ such that
\begin{equation} \label{eq:f4bracket}
\begin{split}
    C &= AB - BA - 4 u \wedge x +4 \lambda^2(v \wedge y)+4\lambda(w \wedge z),\\ 
    r &= Ax - Bu - (vz)^*+(yw)^*, \\
    s &= \lambda(A)y - \lambda(B)v +(wx)^* - (zu)^*, \\
    t &=\lambda^2(A)z - \lambda^2(B)w + (uy)^* - (xv)^*,
\end{split}
\end{equation}
where $a \wedge b = ab^{\top} - ba^{\top}$ for $a, b \in \O$, and where $\lambda$ and $\lambda^2$ are the automorphisms of $\mathfrak{so}(8)$ defined by $\lambda(a \wedge b) = \frac12 L_{b^*}L_{a^*}$ and $\lambda^2(a \wedge b) = \frac12 R_{b^*}R_{a^*}$, for $a, b \in \O, \, a \perp 1$, with $L_c$ and $R_c$ being the left and the right multiplications by $c \in \O$, respectively. 

As the isotropy group $\mathrm{Spin}(9)$ acts transitively on the set of the Lie triple systems in $T_o\O H^2$ tangent to totally geodesic hyperbolic spaces $H^8 \subset \O H^2$ (by the uniqueness part of~\cite[Theorem~1]{Wolf}), we can take $\mathfrak{m}= \{(0,0,v,0) \, | \, v \in \O\}$ and $\mathfrak{m}^\perp= \{(0,0,0,w) \, | \, w \in \O\}$. Then $[\mathfrak{m}, \mathfrak{m}] = \{(\lambda^2(v \wedge y),0,0,0) \, | \, v, y \in \O\} = \mathfrak{so}(8) \oplus \{0\} \oplus \{0\} \oplus \{0\}$, as $\lambda$ is an automorphism. We now need an abelian $8$-dimensional subalgebra $\mathfrak{k} \subset \mathfrak{f}^-_4$ whose projection to $T_o \O H^2$ is $\mathfrak{m}^\perp$ and which, according to  Lemma~\ref{l:tgcomplementary}\eqref{it:bracketorth}, is orthogonal to $[\mathfrak{m}, \mathfrak{m}]$. Then $\mathfrak{k} \subset \{(0,u,0,w) \, | \, u, w \in \O\}$. As $\dim \mathfrak{k} = 8$, we can take $U=(0,u,0,1) \in \mathfrak{k}$ and then for any $z \in \O,\ z\perp 1$, there exists $X=(0,x,0,z) \in \mathfrak{k}$. We have $[U,X] =0$, and so by~\eqref{eq:f4bracket} we obtain $x=zu$ and $\lambda(1 \wedge z) = u \wedge x = u \wedge (zu)$. The latter equation gives $\frac12 L_z = u \wedge (zu)$, and so the element $u \in \O$ has the following property: for any $a \in \O$ and any $z \in \O, z \perp 1$, we have $\frac12 za = \<zu,a\> u - \<u,a\>zu$. Taking a non-zero $z \perp 1$ we obtain $a = 2\<zu,a\> z^{-1}u - 2\<u,a\>u$, which is clearly a contradiction, as we can choose $a \in \O$ which does not lie in the real span of $u$ and $z^{-1}u$.
\end{proof}

\section{%
Number of ignorable coordinates in separating coordinates on  \texorpdfstring{$\C P^n$}{CPn} and  \texorpdfstring{$\C H^n$}{CHn}   } \label{sec:cpn}

\weg{The standard metric of $\C P^n$ is the so-called Fubini--Study metric. 
One obtains it as the quotient metric of the standard sphere $ S^{2n+1} \subset \mathbb{R}^{2n+2}= \C^{n+1}$ by the action of the cirlce $ S^1$ defined by
\[
(z_1, \dots, z_{n+1}) \stackrel{\phi}{\mapsto} (e^{i\phi} z_1, \dots, e^{i\phi} z_{n+1})
\]
where $z_j = x_j + i y_j$ and $(x_1, y_1, \dots, x_{n+1}, y_{n+1})$ are coordinates on $\R^{n+2}$.
This is the flow of the Killing vector field corresponding to the complex structure on $\C ^{n+1}$.  The standard metric of $\C H^n$ can be obtained as the quotient of the restriction of the metric $dx_1^2 +dy_1^2+\cdots+dx_n^2 +dy_n^2-dx_{n+1}^2 -dy_{n+1}^2$ from $\R^{2n+2}$ to the hyperboloid $x_1^2 +y_1^2+\cdots +x_n^2 +dy_n^2-x_{n+1}^2 -y_{n+1}^2$ followed by taking the quotient by the same standard action of $ S^1$. }

As we already know, the spaces $\C P^n$ and $\C H^n$ admit no orthogonal separating coordinates (\cite{gauduchon2020non} and Theorem \ref{th:nondiag}). On the other hand, \cite{BKW1985, WOR1994, ORW1993} provide examples of non-orthogonal separating coordinates. See also \cite{schumm} for another construction of the separation of variables on $\C P^n$ and $\C H ^n$. In these examples, relative to the separating coordinates, the metric has the form~\eqref{eq:coords}, with $\sum_{\alpha,\beta} g_{\alpha\beta}(x) dx_\alpha dx_\beta$ being the metric of constant curvature given in ellipsoidal coordinates, and the number of ignorable variables is exactly $n$. We show that these properties hold for any separation of variables on $\C P^n$ and on $\C H^n$. 

 \begin{theorem} \label{thm:cpn1}
   Let $x_1,\dots,x_{2n}$, with $n\ge 2$, be local separating coordinates on $\C  P^n$ or on $\C H ^n$. 
   
   Then precisely $n$ of them are ignorable, that is, in the notation of Theorem \ref{th:ben}, we have  $r=n$; we denote $x_{n+1}=t_1,\dots, x_{2n}=t_n$. The corresponding vector fields $\tfrac{\partial }{\partial t_1},\dots,    \tfrac{\partial }{\partial t_n}$ are Killing and form a maximal abelian subalgebra of $\mathfrak{su}(n+1)$ in the case  of $\C P^n$, and of $\mathfrak{su}(n,1)$ in the case  of  $\C H ^n$.  
   The coordinates $x_1,\dots, x_n$ are separating coordinates for the metric $g$ 
   of constant positive curvature on the totally geodesic submanifold $\R P^n$ in the case of $\C P^n$, and of constant negative curvature on the totally geodesic submanifold $H^n$ in the case of $\C H^n$.
 \end{theorem}

Note that \cite{Kalnins_book,KalninsMiller1986, BKM2025}, in which separating coordinates for the sphere and for the hyperbolic spaces have been constructed, provide explicit formulas for the metric $g$. 
\weg{ Theorem \ref{thm:cpn1}  will be later complimented by a description of the second block in \eqref{eq:metric}, and transformation formulas from separating coordinates to standard coordinates on $S^{2n+1}\subset \mathbb{R}^{2n}= \C^{n}$. }

\begin{proof}
    For non-orthogonal separating coordinate system on $\C P^n$  or on $\C H^n$, the metric tensor has the form given in~\eqref{eq:coords}, with the submanifolds $\Sigma$ given by $t_i = \mathrm{const}_i$ being totally geodesic, and of constant curvature, by Theorem \ref{th:nondiag}. Moreover, by Lemma~\ref{l:tgcomplementary}\eqref{it:lts}, at a regular point, the orthogonal complement to the tangent space of $\Sigma$ must be a Lie triple system. By~\cite[Theorem~1]{Wolf}, any totally geodesic submanifold of $\C P^n$ is congruent to either a standard $\C P^{k}$ or a standard $\R P^k$, with $k\le n$. Similarly, any totally geodesic submanifold of $\C H^n$ is congruent to either a standard $\C H^{k}$ or a standard $H^k$, with $k\le n$. Moreover, the dimension of a maximal abelian subalgebra of $\mathfrak{su}(n+1)$ is $n$ (the rank of $\mathfrak{su}(n+1)$), and the dimension of a maximal abelian subalgebra of $\mathfrak{su}(n,1)$ is also $n$, by~\cite[Theorem~5.1]{ORWZ1990}. This leaves only two possibilities in the case of $\C P^n$: either $\Sigma$ is a totally real $\R P^n \subset \C P^n$, or $n=2$ and $\Sigma = \C P^1 \subset \C P^2$. Similarly, for $\C H^n$, either $\Sigma$ is a totally real $H^n \subset \C H^n$, or $n=2$ and $\Sigma = \C H^1 \subset \C H^2$.

    We show that both for $\C P^n$ and for $\C H^n$, only the first alternative is possible. We give a proof in the case of $\C P^n$; for $\C H^n$ it is identical, up to obvious changes.

    Suppose that $n = 2$, and that $\Sigma$ is congruent to $\C P^1$. We take
    \begin{equation*}
        \mathfrak{su}(3) = \left\{ \left(
        \begin{array}{ccc}
              a_1 \ri& z_1 & z_2\\
             -\overline{z}_1  & a_2 \ri & z_3\\
             -\overline{z}_2& -\overline{z}_3 & a_3 \ri
        \end{array}
    \right) \, | \, z_1,z_2,z_3 \in \C, \, a_1, a_2, a_3 \in \R, \, a_1+a_2+a_3=0\right\}.
    \end{equation*}
    In the notation of Lemma~\ref{l:tgcomplementary}, at a regular point $o \in \C P^2$, the subspace $T_o\C P^2$ is given by $z_1=a_1=a_2=a_3=0$, and then the subspaces $\mathfrak{m} \subset T_o\C P^2$ and $\mathfrak{m}^\perp \subset T_o\C P^2$, up to isotropy, are given by $z_3=0$ and by $z_2=0$, respectively. The $2$-dimensional abelian subalgebra $\mathfrak{k} \subset \mathfrak{su}(3)$ tangent to the $t$-coordinates is orthogonal to $[\mathfrak{m},\mathfrak{m}]$, by Lemma~\ref{l:tgcomplementary}\eqref{it:bracketorth}, relative to the Killing form of $\mathfrak{su}(3)$, and the projection of $\mathfrak{k}$ to $T_o\C P^2$ equals $\mathfrak{m}^\perp$. But then $\mathfrak{k}$ is spanned by the following two elements:
    \begin{equation*}
         \left(
        \begin{array}{ccc}
              -s \ri& u & 0\\
             -\overline{u}  & 2s \ri & 1\\
             0& -1 & -s \ri
        \end{array}
    \right) \qquad
         \left(
        \begin{array}{ccc}
              -q \ri& v & 0\\
             -\overline{v}  & 2q \ri & \ri\\
             0& \ri & -q \ri
        \end{array}
    \right) \, ,
    \end{equation*}
    for some $u, v \in \C, \, s, q,  \in \R$, and a direct calculation shows that they do not commute.
%
\end{proof}

\section{Conclusion} 
\label{sec:con}

We  solved (the natural analog of) the Eisenhart problem for certain compact rank~1 symmetric spaces.  In  particular, we have shown that $\H P^n$ with $n\ge 2$  and  $\O P^2$ do not admit local separation of variables, and that on  $\C P^n$, all separating  coordinates are  those constructed in \cite{BKW1985}.  We partially solved the  Eisenhart problem for noncompact rank~1 symmetric spaces: we have shown that $\H H ^n$  with $n\ge 2$ and $\O H^2$   do not admit separation of variables and that on  $\C H^n$, any separating coordinates have $n$ ignorable  coordinates. In view of results of \cite{BKW1983,WOR1994, ORW1993, ORWZ1990}, this solves Eisenhart problem  for $\C H^2$  and $\C H^3$. 
 
An algorithm for classifying possible separating variables on the space $\C H^n$, for any given $n$, is in essence given in Theorem~\ref{th:ben}, and includes the following steps. The classification of $n$-dimensional abelian subalgebras $\mathfrak{k} \subset \mathfrak{su}(n,1)$ is given in~\cite[Theorem~5.1]{ORWZ1990}. For any such subalgebra $\mathfrak{k}$, one first constructs the totally geodesic submanifold $\Sigma$ (which is necessarily isometric to $H^n$) orthogonal to the Killing vector fields corresponding to the subalgebra. Quadratic Killing tensor fields on $\C P^n$ are quadratic forms in the Killing vector fields \cite{Eastwood2023, ST1985}. By duality, this is also true for $\C H^n$, so that the quadratic Killing tensor fields on $\C H^n$ are in one-to-one correspondence with the quadratic forms on $\mathfrak{su}(n,1)$. With some aid of computer algebra, one then finds the subspace of all quadratic Killing tensors Poisson-commuting with $\mathfrak{k}$, and then, $n$-dimensional subspaces of that space consisting of pairwise Poisson-commuting quadratic Killing tensors. Next, for such $n$-dimensional subspaces one needs to verify, if the restrictions of its quadratic Killing tensors to the submanifold $\Sigma$ gives the orthogonal separation of variables on $\Sigma$. That is, one needs to check that they have common eigenspaces, and that their Haantjes torsion is zero. This can be reduced to a certain algebraic calculations, requiring working with Gr\"{o}bner bases.  

We note that the classification of polar actions on $\C H^n$ obtained in~\cite{DVK} could be very useful for finding separating variables.
   
Another natural question to address is the separation of variables on symmetric  spaces of higher rank.  Related geometric questions will include the study of the existence of diagonal coordinates on symmetric spaces (see \cite[\textsection 5]{gauduchon2020non} for a list of related open problems), and also the study of abelian subalgebras of isometry algebras of symmetric spaces such that the orthogonal distribution to the span of the values of the corresponding Killing vector fields at a regular point is integrable (and hence, automatically totally geodesic). 

Additional directions for future research are to explore the relation between separation of variables on $\C H^n $ and superintegrable  and multiseparable systems. Recall that the functions $h^{ij}$ from \eqref{eq:coords} can be viewed as a potential adding which to the kinetic energy corresponding to the $N-r$-dimensional metric $g$ does not destroy the integrability and the separability. In the context of separation of variables on $\C P^n$, the functions $h^{ij}$ are constructed from a Cartan subalgebra of $\mathfrak{su}(n+1)$ only. 
In particular, they provide  a superintegrable and multiseparable system. It is easy to check that this system is actually the so-called nondegenerate superintegrable system on the sphere.  In the case of $\C H^n$, there exists $n+2$ pairwise non-conjugate abelian subalgebras of dimension $n$ of $\mathfrak{su}(n,1)$, and each of them gives a natural analog of nondegenerate superintegrable system on the  (real) hyperbolic space.  They were studied in  details for small $n$  in \cite{WOR1994,ORW1993}, with special attention to Cartan subalgebras,  and we plan to extend this study to all values of $n \ge 2$ and all abelian subalgebras of dimension $n$ of $\mathfrak{su}(n,1)$.

\bibliography{sepsym.bib}

\begin{thebibliography}{10}

\bibitem{agafonov}
Sergey~I. Agafonov and Vladimir~S. Matveev.
\newblock Integrable geodesic flows with simultaneously diagonalisable quadratic integrals.
\newblock {\em to appear in Arnold Math. J.}, 2025.

\bibitem{Benenti1990}
S.~Benenti.
\newblock Separation of variables in the geodesic {H}amilton-{J}acobi equation.
\newblock In {\em Symplectic geometry and mathematical physics ({A}ix-en-{P}rovence, 1990)}, volume~99 of {\em Progr. Math.}, pages 1--36. Birkh\"auser Boston, Boston, MA, 1991.

\bibitem{BCR2001}
S.~Benenti, C.~Chanu, and G.~Rastelli.
\newblock Variable separation for natural {H}amiltonians with scalar and vector potentials on {R}iemannian manifolds.
\newblock {\em J. Math. Phys.}, 42(5):2065--2091, 2001.

\bibitem{BCR2002}
S.~Benenti, C.~Chanu, and G.~Rastelli.
\newblock Remarks on the connection between the additive separation of the {H}amilton-{J}acobi equation and the multiplicative separation of the {S}chr\"odinger equation. {I}. {T}he completeness and {R}obertson conditions.
\newblock {\em J. Math. Phys.}, 43(11):5183--5222, 2002.

\bibitem{BF1980}
Sergio Benenti and Mauro Francaviglia.
\newblock The theory of separability of the {H}amilton-{J}acobi equation and its applications to general relativity.
\newblock In {\em General relativity and gravitation, {V}ol. 1}, pages 393--439. Plenum, New York-London, 1980.

\bibitem{Besse}
Arthur~L. Besse.
\newblock {\em Manifolds all of whose geodesics are closed}, volume~93 of {\em Ergebnisse der Mathematik und ihrer Grenzgebiete}.
\newblock Springer-Verlag, Berlin-New York, 1978.

\bibitem{BKM2025}
Alexey~V. Bolsinov, Andrey~Yu. Konyaev, and Vladimir~S. Matveev.
\newblock Orthogonal separation of variables for spaces of constant curvature.
\newblock {\em Forum Math.}, 37(1):13--41, 2025.

\bibitem{BKW1983}
C.~P. Boyer, E.~G. Kalnins, and P.~Winternitz.
\newblock Completely integrable relativistic {H}amiltonian systems and separation of variables in {H}ermitian hyperbolic spaces.
\newblock {\em J. Math. Phys.}, 24(8):2022--2034, 1983.

\bibitem{BKW1985}
C.~P. Boyer, E.~G. Kalnins, and P.~Winternitz.
\newblock Separation of variables for the {H}amilton-{J}acobi equation on complex projective spaces.
\newblock {\em SIAM J. Math. Anal.}, 16(1):93--109, 1985.

\bibitem{BG}
Robert~B. Brown and Alfred Gray.
\newblock Riemannian manifolds with holonomy group {${\rm S}pin$} (9).
\newblock In {\em Differential geometry (in honor of {K}entaro {Y}ano)}, pages 41--59. Kinokuniya, Tokyo, 1972.

\bibitem{ORW1993}
M.~A. del Olmo, M.~A. Rodriguez, and P.~Winternitz.
\newblock Integrable systems based on {${\rm {SU}}(p,q)$} homogeneous manifolds.
\newblock {\em J. Math. Phys.}, 34(11):5118--5139, 1993.

\bibitem{ORWZ1990}
M.~A. del Olmo, M.~A. Rodr\'iguez, P.~Winternitz, and H.~Zassenhaus.
\newblock Maximal abelian subalgebras of pseudounitary {L}ie algebras.
\newblock {\em Linear Algebra Appl.}, 135:79--151, 1990.

\bibitem{DiazRamos}
Jos\'e{}~Carlos D\'iaz-Ramos.
\newblock Polar actions on symmetric spaces.
\newblock In {\em Recent trends in {L}orentzian geometry}, volume~26 of {\em Springer Proc. Math. Stat.}, pages 315--334. Springer, New York, 2013.

\bibitem{DVK}
Jos\'e{}~Carlos D\'iaz~Ramos, Miguel Dom\'inguez~V\'azquez, and Andreas Kollross.
\newblock Polar actions on complex hyperbolic spaces.
\newblock {\em Math. Z.}, 287(3-4):1183--1213, 2017.

\bibitem{Eastwood2023}
Michael Eastwood.
\newblock Killing tensors on complex projective space.
\newblock {\em ArXiv}, https://arxiv.org/abs/2309.00589, 2023.

\bibitem{Eisenhart1934}
Luther~P. Eisenhart.
\newblock Separable systems of {S}t{\"a}ckel.
\newblock {\em Ann. of Math. (2)}, 35(2):284--305, 1934.

\bibitem{gauduchon2020non}
Paul Gauduchon and Andrei Moroianu.
\newblock Non-existence of orthogonal coordinates on the complex and quaternionic projective spaces.
\newblock {\em Journal of Geometry and Physics}, 155:103770, 2020.

\bibitem{Gorodski}
Claudio Gorodski.
\newblock Polar actions on compact symmetric spaces which admit a totally geodesic principal orbit.
\newblock {\em Geom. Dedicata}, 103:193--204, 2004.

\bibitem{Gray}
Alfred Gray.
\newblock A note on manifolds whose holonomy group is a subgroup of {${\rm Sp}(n)\cdot {\rm Sp}(1)$}.
\newblock {\em Michigan Math. J.}, 16:125--128, 1969.

\bibitem{Helgason}
Sigurdur Helgason.
\newblock {\em Differential geometry, {L}ie groups, and symmetric spaces}, volume~80 of {\em Pure and Applied Mathematics}.
\newblock Academic Press, New York-London, 1978.

\bibitem{Kalnins_book}
E.~G. Kalnins.
\newblock {\em Separation of variables for {R}iemannian spaces of constant curvature}, volume~28 of {\em Pitman Monographs and Surveys in Pure and Applied Mathematics}.
\newblock Longman Scientific \& Technical, Harlow; John Wiley \& Sons, Inc., New York, 1986.

\bibitem{KalninsMiller1986}
E.~G. Kalnins and W.~Miller, Jr.
\newblock Separation of variables on {$n$}-dimensional {R}iemannian manifolds. {I}. {T}he {$n$}-sphere {$S_n$} and {E}uclidean {$n$}-space {${\bf R}^n$}.
\newblock {\em J. Math. Phys.}, 27(7):1721--1736, 1986.

\bibitem{KMR1984}
E.~G. Kalnins, W.~Miller, Jr., and G.~J. Reid.
\newblock Separation of variables for complex {R}iemannian spaces of constant curvature. {I}. {O}rthogonal separable coordinates for {${\rm S}\sb {n{\bf C}}$} and {${\rm E}\sb {n{\bf C}}$}.
\newblock {\em Proc. Roy. Soc. London Ser. A}, 394(1806):183--206, 1984.

\bibitem{KalninsMiller1981}
E.~G. Kalnins and Willard Miller, Jr.
\newblock Killing tensors and nonorthogonal variable separation for {H}amilton-{J}acobi equations.
\newblock {\em SIAM J. Math. Anal.}, 12(4):617--629, 1981.

\bibitem{KKM18}
Ernest~G Kalnins, Jonathan~M Kress, and Willard Miller.
\newblock {\em Separation of variables and Superintegrability: The symmetry of solvable systems}.
\newblock IOP Publishing, 2018.

\bibitem{Kollross}
Andreas Kollross.
\newblock Octonions, triality, the exceptional {L}ie algebra ${F}_4$ and polar actions on the {C}ayley hyperbolic plane.
\newblock {\em Internat. J. Math.}, 31(7):2050051, 28, 2020.

\bibitem{Levi-Civita1904}
T.~Levi-Civita.
\newblock Sulla integrazione della equazione di {Hamilton}-{Jacobi} per separazione di variabili.
\newblock {\em Math. Ann.}, 59:383--397, 1904.

\bibitem{MN2024}
Vladimir~S. Matveev and Yuri Nikolayevsky.
\newblock Quadratic {K}illing tensors on symmetric spaces which are not generated by {K}illing vector fields.
\newblock {\em C. R. Math. Acad. Sci. Paris}, 362:1043--1049, 2024.

\bibitem{MN2025}
Vladimir~S. Matveev and Yuri Nikolayevsky.
\newblock Killing tensors on reducible spaces.
\newblock {\em Manuscripta Math.}, 176(1):Paper No. 3, 9, 2025.

\bibitem{schumm}
Jan Schumm.
\newblock Geodesic flows of c-projectively equivalent metrics are quantum integrable.
\newblock {\em Eur. J. Math.}, 8(4):1566--1601, 2022.

\bibitem{disser}
P.~St\"ackel.
\newblock Die integration der {H}amilton-{J}acobischen differentialgleichung mittelst separation der variablen.
\newblock {\em Habilitationsschrift, Universit\"at Halle}, 1891.

\bibitem{Staeckel1893}
P.~St{\"a}ckel.
\newblock Sur une classe de probl{\`e}mes de dynamique.
\newblock {\em C. R. Acad. Sci., Paris}, 116:485--487, 1893.

\bibitem{Staeckel1897}
P.~St{\"a}ckel.
\newblock Ueber quadratische {Integrale} der {Differentialgleichungen} der {Dynamik}.
\newblock {\em Annali di Mat. (2)}, 25:55--60, 1897.

\bibitem{ST1985}
Takeshi Sumitomo and Kwoichi Tandai.
\newblock On the centralizer of the {L}aplacian of {${\bf P}\sb n({\bf C})$} and the spectrum of complex {G}rassmann manifold {${\bf G}\sb {2,n-1}({\bf C})$}.
\newblock {\em Osaka J. Math.}, 22(1):123--155, 1985.

\bibitem{WOR1994}
P.~Winternitz, M.~A. del Olmo, and M.~A. Rodr\'iguez.
\newblock Maximal abelian subgroups of {${\rm {SU}}(p,q)$} and integrable {H}amiltonian systems.
\newblock {\em Noncompact {L}ie groups and some of their applications ({S}an {A}ntonio, {TX}, 1993), NATO Adv. Sci. Inst. Ser. C: Math. Phys. Sci.}, 429:181--198, 1994.

\bibitem{Wolf}
Joseph~A. Wolf.
\newblock Elliptic spaces in {G}rassmann manifolds.
\newblock {\em Illinois J. Math.}, 7:447--462, 1963.

\end{thebibliography}
\bibliographystyle{plain}

\end{document}